\documentclass[reqno,a4paper,10pt]{amsart} 
\usepackage{amssymb} 
\usepackage{latexsym} 
\usepackage{amsmath} 
\usepackage{esint} 
\usepackage[utf8x]{inputenc}
\usepackage{todonotes}
\usepackage{MnSymbol}
\usepackage[colorlinks,pdfpagelabels,pdfstartview = FitH,bookmarksopen = true,bookmarksnumbered = true,linkcolor = blue,plainpages = false,hypertexnames = false,citecolor = red] {hyperref} 
\allowdisplaybreaks

\newtheorem{theorem}{Theorem}[section]

\newtheorem{lemma}[theorem]{Lemma}

\theoremstyle{definition} 
\newtheorem{definition}[theorem]{Definition}

\numberwithin{equation}{section}
\newcommand{\R}{{\mathbb R}}

\newcommand{\U}{{\mathcal{U}}}

\newcommand{\Si}{\ensuremath{\mathcal{S}}} 
\newcommand{\Bo}{\ensuremath{\mathcal{B}}} 
\newcommand{\T}{\ensuremath{\mathcal{T}}}

\newcommand{\grad}{\nabla}

\newcommand{\de}{\dif}
\newcommand{\dx}{\de x} 
\newcommand{\dt}{\de t} 
 
\newcommand{\dxdt}{\dx\dt}
\DeclareMathOperator*{\essliminf}{ess\,lim\,inf}   \DeclareMathOperator*{\capacity}{cap} \DeclareMathOperator*{\supp}{supp}

\newcommand{\dif}[0]{\ensuremath{\,\mathrm{d}}} 
\newcommand{\norm}[1]{\ensuremath{\Vert #1 \Vert}} 
 
\newcommand{\inprod}[1]{\ensuremath{\langle #1 \rangle}} 
 
\newcommand{\dinprod}[1]{\ensuremath{\langle\langle #1 \rangle\rangle}} 
 
\newcommand{\abs}[1]{\ensuremath{\vert #1 \vert}} 
 
\newcommand{\spcdot}{\,\cdot\,} 
\newcommand{\lscss}{semicontinuous supersolution}

\makeatletter
\providecommand\@dotsep{5}
\renewcommand{\listoftodos}[1][\@todonotes@todolistname]{%
  \@starttoc{tdo}{#1}}
\makeatother
\hypersetup{
	linktoc=page
}

\begin{document} 
\title[A comparison principle for the PME]{A comparison principle for
  the porous medium equation and its consequences}

\address{Benny Avelin,
Aalto University, 
Institute of Mathematics,
P.O. Box 11100, 
FI-00076 Aalto, 
Finland}
\address{Benny Avelin,
Department of Mathematics, 
Uppsala University,
S-751 06 Uppsala, 
Sweden} 
\email{benny.avelin@math.uu.se}

\address{Teemu Lukkari\\
Department of Mathematics and Statistics, University of Jyväskylä\\
P.O. Box 35, 40014 Jyväskylä, Finland} \email{teemu.j.lukkari@jyu.fi} 
\author[Avelin]{Benny Avelin}
\author[Lukkari]{Teemu Lukkari} 
 \thanks{The first author was partially supported by Academy of Finland, project \#259224, and by the Swedish Research Council, \#637-2014-6822.} \thanks{A part of the research reported in this work was done during the authors' stay at the Institute Mittag-Leffler (Djursholm, Sweden).}

\begin{abstract}
	
  We prove a comparison principle for the porous medium equation in
  more general open sets in $\R^{n+1}$ than space-time cylinders. We
  apply this result in two related contexts: we establish a connection
  between a potential theoretic notion of the obstacle problem and a
  notion based on a variational inequality. We also prove the basic
  properties of the PME capacity, in particular that there exists a
  capacitary extremal which gives the capacity for compact sets.
\end{abstract}

\maketitle

\section{Introduction} 

We study the porous medium equation (PME for short)
\begin{equation}\label{eq:pme-intro}
	\frac{ 
	\partial u}{ 
	\partial t}-\Delta u^m=0\,, 
\end{equation}
where $m>1$. This equation is an important prototype of a nonlinear
parabolic equation. The equation is degenerate, meaning that the
modulus of ellipticity vanishes when the solution is zero.  The name
stems from modeling the flow of a gas in a porous medium: the
continuity equation, Darcy's law, and an equation of state for the gas
lead to \eqref{eq:pme-intro} for the density of the gas, after scaling
out various physical constants.  For more information about
this equation, including numerous further references, we refer to the
monographs \cite{DaskalopoulosKenig} and \cite{VazquezBook}.

The comparison principle is a fundamental tool in the theory of
elliptic and parabolic equations. In particular, it can be used to
define a class of supersolutions which is the counterpart for
superharmonic functions in classical potential theory: we call a
function a \lscss{}, if it satisfies the comparison principle with
respect to continuous solutions.  The definition is due to
F. Riesz~\cite{Riesz}, and it makes the development of a nonlinear
potential theory feasible.

The comparison principle for parabolic equations is usually formulated
for space-time cylinders, meaning sets of the form
$\Omega_T=\Omega\times (0,T)$.  The boundary values are then taken
over \emph{the parabolic boundary}, where only the initial and lateral
boundaries are taken into account.  However, one often encounters
situations where one would like to apply the comparison principle in
sets which are not space-time cylinders. Thus our main objective is to
establish a comparison principle for the PME in more general open sets
in $\R^{n+1}$. Such a result is occasionally called the elliptic
comparison principle, in reference to the fact that the time variable
no longer has a special role. Moreover, the elliptic comparison principle can be used to develop the Perron method in general space-time domains, see \cite{BBGP,KL}.
We also present two applications where
such a comparison principle is indispensable.

For the heat equation, when $m=1$, one may add constants to solutions.
A comparison principle for general open sets then follows from the
space-time cylinder case by a straightforward exhaustion argument. For
the PME, there is a comparison principle over cylindrical domains, but
adding constants is no longer possible. Our idea for circumventing
this difficulty is to multiply one of the functions being compared by
a constant close to one. The modified function is no longer a
solution, but it still satisfies the PME with an error-term on
the right hand side. The error-term vanishes as the
multiplicative constant tends to one. The comparison principle for the
original functions then follows by the usual duality proof, modified
to account for the error-term. Our argument yields a comparison
principle for open sets of the form $\Omega_T\setminus K$, where $K$
is a compact set.

As the first application, we consider the obstacle problem. Roughly
speaking, this amounts to finding a solution to a PDE subject to the
constraint that the solution stays above a given function, the
obstacle. Here we use a potential theoretic method for solving the
problem: we define the solution to the obstacle problem to be the
infimum of all supersolutions lying above the obstacle (réduite). For smooth enough obstacles the réduite is the smallest supersolution above the obstacle. The concept of réduite is standard in classical potential theory, and it has been
utilized in a nonlinear parabolic context in \cite{LP}.  Existence and
uniqueness follow in a straightforward manner, at least for continuous
obstacles. However, the relation between the smallest supersolution
and the variational solutions to obstacle problems constructed in
\cite{BLS} is not obvious. In this direction, we prove that the
smallest supersolution is also a variational solution for sufficiently
smooth obstacles. This follows from two facts. First, we prove that
the smallest supersolution can always be approximated by variational
solutions. Second, the notion of variational solution is stable with
respect to the convergence of the obstacles in certain norms, see
\cite{BLS}. The converse of this, i.e. whether a variational solution
agrees with the smallest supersolution, remains a very interesting
open problem.

The second application is a notion of parabolic capacity for the PME.
This concept is defined via a measure data problem, as in \cite{KKKP}
for the parabolic $p$-Laplacian.  See also \cite{W, WB} and the references therein for the
capacity for the heat equation.  We prove the basic properties of the
capacity related to the PME, such as countable subadditivity and the
existence of the capacitary extremal of a compact set. 
Our comparison principle plays a key role in the latter argument.

The paper is organized as follows. In Section \ref{sec:weaksuper}, we
recall the necessary background material, in particular various
notions of supersolutions. Section \ref{sec:comp} contains the proof of
the comparison principle, and Section \ref{sec:ost} is concerned with
the obstacle problem. Finally, the basic properties of capacity are
proved in Section \ref{sec:cap}.

\section{Weak supersolutions and semicontinuous
  supersolutions} \label{sec:weaksuper}

Let $\Omega$ be an open and bounded subset of $\R^n$, and let $0<t_1<t_2<T$. We use the notation $\Omega_T=\Omega\times(0,T)$ and $U_{t_1,t_2}=U\times (t_1,t_2)$, where $U\subset\Omega$ is open. The parabolic boundary $ 
\partial_p U_{t_1,t_2}$ of a space-time cylinder $U_{t_1,t_2}$ consists of the initial and lateral boundaries, i.e. 
\begin{displaymath}
	\partial_p U_{t_1,t_2}=(\overline{U}\times\{t_1\}) \cup ( 
	\partial U\times [t_1,t_2])\,. 
\end{displaymath}
The notation $U_{t_1,t_2}\Subset\Omega_T$ means that the closure $\overline{U_{t_1,t_2}}$ is compact and $\overline{U_{t_1,t_2}}\subset\Omega_T$.

We use $H^1(\Omega)$ to denote the usual Sobolev space, the space of
functions $u$ in $L^2(\Omega)$ such that the weak gradient exists and
also belongs to $L^2(\Omega)$. The norm of $H^1(\Omega)$ is defined by
\begin{displaymath}
	\norm{u}_{H^1(\Omega)}^2=\norm{u}_{L^2(\Omega)}^2+\norm{\nabla u}_{L^2(\Omega)}^2\,. 
\end{displaymath}
The Sobolev space with zero boundary values, denoted by $H^{1}_0(\Omega)$, is the completion of $C^{\infty}_0(\Omega)$ with respect to the norm of $H^1(\Omega)$.

The parabolic Sobolev space $L^2(0,T;H^1(\Omega))$ consists of measurable functions $u:\Omega_T\to[-\infty,\infty]$ such that $x\mapsto u(x,t)$ belongs to $H^1(\Omega)$ for almost all $t\in(0,T)$, and 
\begin{displaymath}
	\int_{\Omega_T}\abs{u}^2+\abs{\nabla u}^2\dx\dt<\infty\,. 
\end{displaymath}
The definition of $L^2(0,T;H^{1}_0(\Omega))$ is identical, apart from the requirement that $x\mapsto u(x,t)$ belongs to $H^{1}_0(\Omega)$. We say that $u$ belongs to $L^2_{loc}(0,T;H^{1}_{loc}(\Omega))$ if $u\in L^2(t_1,t_2;H^1(U))$ for all $U_{t_1,t_2}\Subset\Omega_T$.

Supersolutions to the porous medium equation are defined in the weak sense in the parabolic Sobolev space. 
\begin{definition}
	\label{def:local-weak} A nonnegative function $u:\Omega_T\to\R$ is a \emph{weak supersolution} of the equation 
	\begin{equation}
		\label{eq:pme} \frac{ 
		\partial u}{ 
		\partial t}-\Delta u^m=0 
	\end{equation}
	in $\Omega_T$, if $u^m\in L^2_{loc}(0,T;H^{1}_{loc}(\Omega))$ and 
	\begin{equation*} 
		\int_{\Omega_T}-u\frac{ 
		\partial\varphi}{ 
		\partial t} +\nabla u^m\spcdot\nabla\varphi\dx\dt\geq 0 \,,
	\end{equation*}
	for all positive, smooth test functions $\varphi$ compactly supported in $\Omega_T$. The definition of \emph{weak subsolutions} is similar; the inequality is simply reversed. \emph{Weak solutions} are defined as functions that are both super- and subsolutions. 
\end{definition}

Weak solutions are locally H\"older continuous, after a possible redefinition on a set of measure zero. See \cite{DK}, \cite{DaskalopoulosKenig}, \cite{DiBeFried}, \cite{VazquezBook}, or \cite{Kiinalaiset}.

We have also the following class of supersolutions. 
\begin{definition}
	\label{def:viscosity-supersols} A function $u:\Omega_T\to [0,\infty]$ is a \emph{\lscss }, if 
	\begin{enumerate}
		\item $u$ is lower semicontinuous, 
		\item $u$ is finite in a dense subset of $\Omega_T$, and 
		\item the following parabolic comparison principle holds: Let $U_{t_1,t_2}\Subset\Omega$, and let $h$ be a solution to \eqref{eq:pme} which is continuous in $\overline{U_{t_1,t_2}}$. Then, if $h\leq u$ on $ 
		\partial_p U_{t_1,t_2}$, $h\leq u$ also in $U_{t_1,t_2}$. 
	\end{enumerate}
\end{definition}
Note that a \lscss{} is defined in every point.  Every weak
supersolution is a \lscss{} provided that a proper pointwise
representative is chosen. This is a consequence of the following
lemma.
\begin{lemma}[\cite{AL}] \label{thm:lsc:intro} Let $u$ be a
  nonnegative weak supersolution to the porous medium equation in
  $\Omega\times(t_1,t_2)$. Then $u$ has a lower semicontinuous
  representative.
\end{lemma}
In the other direction, a \emph{bounded} \lscss{} is also a weak
supersolution, as shown in \cite{KinnunenLindqvist2}. If unbounded
functions are allowed, then the class of \lscss s is strictly larger,
since the Barenblatt solution is a \lscss{}, but is not a weak
supersolution, see \cite{KinnunenLindqvist2}.

\begin{lemma}
	[\cite{KinnunenLindqvist2}] \label{lem:caccioppoli} Let $u$ be a weak supersolution such that $\abs{u}\leq M<\infty$. Then 
	\begin{displaymath}
		\iint_{\Omega_T}\eta^2\abs{\nabla u^m }^2 \dxdt \leq 16M^{2m}T\int_{\Omega}\abs{\nabla\eta}^2\dx +6M^{m+1}\int_{\Omega}\eta^2 \dx \,,
	\end{displaymath}
	for all nonnegative functions $\eta\in C_0^\infty(\Omega)$. 
\end{lemma}

An application of the Riesz representation theorem shows that for each weak supersolution $u$, there exists a positive Radon measure $\mu_u$ such that 
\begin{displaymath}
	\iint_{\Omega_\infty}-u\frac{\partial\varphi}{\partial t} +\nabla u^m\spcdot\nabla\varphi\dx\dt=\int_{\Omega_\infty}\varphi \de \mu_{u}
\end{displaymath}
for all smooth compactly supported functions $\varphi$.  This is the
\emph{Riesz measure} of $u$. The integrals on the left hand side do
not depend on the particular point-wise representative of a
supersolution. Thus a weak supersolution $u$ and its lower
semicontinuous regularization $\widehat{u}$ have the same Riesz
measures.
\begin{lemma}
	\label{ComparisonPrinciple} If $u$ and $v$ are weak supersolutions in $\Omega_\infty$, $u,v = 0$ on $
	\partial_p \Omega_\infty$, $u^m, v^m\in
        L^2(0,\infty;H^1_0(\Omega))$, and $\mu_v \leq \mu_u$, then $v
        \leq u$ a.e. in $\Omega_\infty$.
\end{lemma}
\begin{proof}
  Let $\varphi\in C^\infty_0(\Omega_\infty)$ be nonnegative.  By
  subtracting the equations satisfied by $u$ and $v$ and using the
  assumption about the measures, we have
  \begin{equation*}
    \int_{\Omega_\infty}-(u-v)\varphi_t+\nabla(u^m-v^m)\spcdot \nabla
    \varphi\dx \dt=\int_{\Omega_\infty}\varphi\de
    \mu_u-\int_{\Omega_\infty}\varphi\de \mu_v\geq 0\,.
  \end{equation*}
  By a standard approximation argument using the fact that $u^m,
  v^m\in L^2(0,\infty;H^1_0(\Omega))$, we may also take the test
  functions $\varphi\in C^\infty(\Omega_\infty)$ so that $\varphi=0$
  on the lateral boundary $\partial\Omega\times(0,\infty)$. We apply
  the Green's formula to get
  \begin{equation*}
    \int_{\Omega_\infty}-(u-v)\varphi_t-(u^m-v^m) \Delta
    \varphi\dx \dt\geq 0\,.
  \end{equation*}
  The fact that $v\leq u$ follows from this inequality by repeating
  the standard duality proof for the comparison principle for the PME,
  see e.g. \cite[Lemma 5]{DK}, \cite[Theorem
  1.1.1]{DaskalopoulosKenig}, or \cite[Theorem 6.5]{VazquezBook}.
\end{proof}

\begin{lemma}
	\label{Convergence} Let $u_i$, $i=1,2,\ldots$ is a uniformly bounded sequence of weak supersolutions in $\Omega_\infty$ such that $u_i \to u$ a.e. in $\Omega_\infty$. Then $u$ is a weak supersolution in $\Omega_\infty$ and 
	\begin{equation*}
		\lim_{i \to \infty} \int_{\Omega_\infty} \phi d \mu_{u_i} = \int_{\Omega_\infty} \phi d \mu_u\,, 
	\end{equation*}
	for every $\phi \in C_0^\infty(\Omega_\infty)$. 
\end{lemma}
\begin{proof}
	Due to the uniform bound on the functions $u_i$, it easily follows that 
	\begin{equation} \label{weakineq}
		\int_{\Omega_\infty} -u\frac{
		\partial \phi}{
		\partial t}-u^m \Delta \phi \dx\dt\geq 0\,. 
	\end{equation}
	An application of Lemma \ref{lem:caccioppoli} on each $u_i$
        implies that $\nabla u^m\in L^2_{loc}(\Omega_\infty)$. This, 
        together with \eqref{weakineq} yields that $u$ is a weak
        supersolution. The claim about the measures follows from the
        computation
	\begin{align*}
		\lim_{i \to \infty} \int_{\Omega_\infty} \phi d \mu_{u_i}=&\lim_{i\to\infty} \int_{\Omega_\infty}-u_i\frac{
		\partial \phi}{
		\partial t}+\nabla u_i^m \spcdot\nabla\phi\dx\dt\\
		= &\int_{\Omega_\infty}-u\frac{
		\partial \phi}{
		\partial t}+\nabla u^m \spcdot\nabla\phi\dx\dt\\
		=& \int_{\Omega_\infty} \phi d \mu_u.\qedhere 
	\end{align*}
\end{proof}

We will frequently use the following characterization of the weak convergence of measures. See \cite[Theorem 1, p. 54]{EG} for the proof. 
\begin{theorem}
	\label{thm:weak-meas} Let $\mu$ and $\mu_k$, $k=1,2,3,\ldots,$ be Radon measures on $\R^n$. Then the following statements are equivalent. 
	\begin{enumerate}
		\item For all compactly supported smooth functions $\phi$, one has 
		\begin{displaymath}
			\lim_{k\to \infty}\int_{\R^n}\phi d\mu_k=\int_{\R^n}\phi d\mu\,. 
		\end{displaymath}
		\item For all compact sets $K$, one has 
		\begin{displaymath}
			\limsup_{k\to \infty}\mu_k(K)\leq\mu(K)\,. 
		\end{displaymath}
              \item 
		For all open sets $U$, one has 
		\begin{displaymath}
			\mu(U)\leq \liminf_{k\to \infty}\mu_k(U)\,. 
		\end{displaymath}
	\end{enumerate}
\end{theorem}

\section{A comparison principle}
\label{sec:comp}

The core of our  arguments is a suitable form of the comparison principle, which we will prove in this section. We will work extensively with finite unions of space-time cylinders, so we begin by introducing some notation for such sets. For space-time cylinders $ U_{t_1,t_2}=U\times(t_1,t_2)$, we denote the lateral boundary by 
\begin{displaymath}
	\Si(U_{t_1,t_2})=\partial U\times (t_1,t_2)\,. 
\end{displaymath}
For a cylinder the definition of the parabolic boundary is standard, but for finite unions of space time cylinders we will recall the definitions. The lateral boundary of a finite union of space time cylinders $U^i_{t_1^i, t_2^i}$ is then given by 
\newcommand{\Ui}{U^i_{t_1^i, t_2^i}} 
\begin{equation*}
	\Si (\cup \Ui) := (\cup \Si(\Ui)) \setminus (\cup \Ui)\,. 
\end{equation*}
We also denote the tops of $\cup \Ui$ by 
\begin{equation*}
	\T (\cup \Ui) := (\cup \overline{U^i} \times \{t_2^i\}) \setminus (\cup \Ui)\,,
\end{equation*}
and the bottoms similarly as 
\begin{equation*}
	\Bo (\cup \Ui) := (\cup \overline{U^i} \times \{t_1^i\}) \setminus (\cup \Ui)\,. 
\end{equation*}
Thus the parabolic boundary of $Q = \cup \Ui$ is $\Si(Q) \cup \Bo(Q)$ and the parabolic boundary of backwards in time equations becomes $\Si(Q) \cup \T(Q)$.

We want to use the very weak (i.e. distributional) formulation of the
porous medium equation, so we consider smooth test functions $\phi \in
C^\infty(Q)$ where $Q = \cup \Ui$, such that $\phi = 0$ on
$\Si(Q)$. Note that the gradient of $\phi$ does not necessarily vanish
on $\Si(Q)$.  In the following we will always work with $\Omega$ a
smooth domain. Let us now write the PME in terms of the above class of
test functions. Assume at first that $\phi$ has compact support in
space. Then a standard approximation argument shows that we may write
the definition of weak solutions as
\begin{equation*}
	\int_{Q}\big [-u\phi_t +  \nabla u^m \spcdot\nabla \phi \big ] \dxdt 
	+\int_{\T(Q)} u \phi \dx - \int_{\Bo(Q)} u \phi \dx
	=0\,.
\end{equation*}
After this, we may pass from compactly supported test functions to test functions vanishing on the sides $\Si(Q)$, since $u$ and $\nabla u^m$ are in $L^2(Q)$. Now, apply Green's formula, which is
justified by the usual trace theorem, to get 
\begin{multline} \label{weakintpart}
	\int_{\T(Q)} u \phi \dx - \int_{\Bo(Q)} u \phi \dx +\int_Q \big [-u \phi_t  -  u^m \Delta \phi \big ] \dxdt \\
	+ \int_{\Si(Q)} u^m \partial_n \phi\, d \sigma \dt=0\,.
\end{multline}
A similar argument can be carried out for weak supersolutions and subsolutions. In these cases, we get the appropriate inequalities in the final form.  This formulation will be our starting point in the proof of the comparison principle.

\begin{theorem}
  \label{Comparison1} 
  Let $K$ be a compact set in $\Omega_T$ where $\Omega$ is a smooth
  domain, let $u$ be a nonnegative upper semicontinuous function which
  is a continuous weak supersolution in in $\Omega_T \setminus K$ and
  satisfies $u^m\in L^2(0,T; H^1(\Omega))$.  Let $v$ be a non-negative
  lower semicontinuous function which is a weak supersolution in
  $\Omega_T\setminus K$ and satisfies $v^m\in L^2(0,T;H^1(\Omega))$,
  $v > 0$ on $K$ and $u \leq v$ on $K \cup \partial_p \Omega_T$. Then
  $u \leq v$ in $\Omega_T$.
\end{theorem}
\begin{proof}
	We let $\epsilon > 0$, and denote 
	\begin{equation*}
		D_\epsilon = \left \{(x,t) \in \Omega_T: \frac{u}{1+\epsilon} \geq v \right \}\,.
	\end{equation*}
	The function $\frac{u}{1+\epsilon}-v$ is upper semicontinuous, so that the set $D_\epsilon$ is closed in $\Omega_T$. Moreover, the set $D_\epsilon$ does not intersect the set $K$, since $u\leq v$ on $K$ and $\inf_{K} v > 0$. Since $K$ is compact, there is a positive distance between $D_\epsilon$ and $K$. Thus we can cover $D_\epsilon$ with a finite collection of space time cylinders not intersecting $K$. Denote the covering set by $\hat D^F_{\epsilon}$, and note that since $D_{\epsilon_1} \subset D_{\epsilon_2}$ for $\epsilon_1 > \epsilon_2$ we may choose the coverings for different values of $\epsilon$ so that $\hat D^F_{\epsilon_1}\subset \hat D^F_{\epsilon_2}$. Let then $D^F_{\epsilon} = \Omega_T \cap D^F_{\epsilon}$. The set $D^F_{\epsilon}$ is still a finite union of space-time cylinders, and the function $u$ is a weak solution in $D^F_{\epsilon}$.

	Let $u_\epsilon = \frac{u}{1+\epsilon}$. We want to compare $u_\epsilon$ with $v$ in $D^F_{\epsilon}$. To this end, note first that $u_\epsilon < v$ on $
	\partial D^F_{\epsilon}$. Further, the function $u_\epsilon$ is a solution to 
	\begin{equation*}
		(u_\epsilon)_t - \Delta (u_\epsilon)^m = f:= \frac{(1+\epsilon)^{m-1} - 1}{(1+\epsilon)^{m}}\Delta u^m\,,
	\end{equation*}
	interpreted in the sense of distributions. To see this, we compute 
	\begin{align} \label{uepsilon}
		[u_\epsilon]_t-\Delta [u_\epsilon]^m=& [u_\epsilon-u]_t-\Delta[u_\epsilon^m-u^m] \\
		=&\left(\frac{1}{1+\epsilon}-1\right)u_t -\left(\frac{1}{(1+\epsilon)^m}-1\right)\Delta u^m \notag \\
		=&\left(\frac{1}{1+\epsilon}-\frac{1}{(1+\epsilon)^{m}}\right)\Delta u^m =\frac{(1+\epsilon)^{m-1} - 1}{(1+\epsilon)^{m}}\Delta u^m\,. \notag
	\end{align}
	
	We aim at adapting the proof of the comparison principle for
        the PME, see e.g. \cite{DK}. To proceed let
	\begin{equation*}
		D^F_{\epsilon,s} = \{ (x,t)\in D^F_{\epsilon}:t\leq s\}\,, 
	\end{equation*}
	and take positive functions $\phi\in C^{\infty}_0(D^F_{\epsilon,s})$, and $\psi \in C^{\infty}(D^F_{\epsilon,s})$ which vanishes on $\Si(D^F_{\epsilon,s})$, $\partial_n \psi\leq 0$ on $\Si(D^F_{\epsilon,s})$ and so that $\psi$ equals $\phi$ on $\T(D^F_{\epsilon,s})$. Denote also $b = u_{\epsilon}-v$ for brevity. Since $b$ is negative and consequently also $u_\epsilon^m-v^m$ on $\partial D^F_{\epsilon}$, we have from \eqref{weakintpart} and \eqref{uepsilon}, since $\partial_n \psi \leq 0$ on $\Si(D^F_{\epsilon})$
	\begin{align*}
		\int_{D^F_{\epsilon,s}} f \psi \dxdt &= \int_{\T(D^F_{\epsilon,s})} b \phi \dx - \int_{\Bo(D^F_{\epsilon,s})} b \psi \dx - \int_{D^F_{\epsilon,s}} b \psi_t \dxdt\\
		&- \int_{D^F_{\epsilon,s}} [u^m_{\epsilon}-v^m] \Delta \psi \dxdt + \int_{\Si(D^F_{\epsilon,s})} [u^m_{\epsilon}-v^m] 
		\partial_n \psi\, d \sigma \dt \\
		&\geq \int_{\T(D^F_{\epsilon,s})} b \phi \dx - \int_{D^F_{\epsilon,s}} b \psi_t \dxdt - \int_{D^F_{\epsilon,s}} [u^m_{\epsilon}-v^m] \Delta \psi \dxdt \,.
	\end{align*}
	We can rewrite this as 
	\begin{align} \label{goodinequality}
		\int_{\T(D^F_{\epsilon,s})} b \phi \dx \leq \int_{D^F_{\epsilon,s}} b (\psi_t + a \Delta \psi) \dxdt + \int_{D^F_{\epsilon,s}} f \psi \dxdt\,,
	\end{align}
	where 
	\begin{equation*}
		a = 
		\begin{cases}
			\frac{u^m_{\epsilon}-v^m}{u_\epsilon-v}, & \text{if }u_\epsilon\not = v\,,\\
			0, & \text{if }u_\epsilon=v\,.
		\end{cases}
	\end{equation*}
	Next we use a regularization to make the term $\psi_t + a \Delta \psi$ small in the above inequality. To do this let $a_k$, $k=1,2,\ldots,$ be smooth functions in $\overline{D^F_{\epsilon,s}}$ such that 
	\begin{equation*}
		\frac{1}{k}\leq a_k \leq k\,,
	\end{equation*}
	and 
	\begin{equation} \label{appconvergence}
		\int_{\overline{D^F_{\epsilon,s}}} \frac{(a_k-a)^2}{a_k} \dxdt \to 0 \quad \text{as} \quad k \to \infty \,.
	\end{equation}
	We replace the function $\psi$ in \eqref{goodinequality} by the solution $\psi_k$ to the following boundary value problem 
	\begin{equation} \label{backwards_bvp}
		\begin{cases}
			u_t + a_k \Delta u = 0\,, &\text{ in } D^F_{\epsilon,s}\,,\\
			u(x,s) = \phi(x,s)\,, &\text{ on } \T(D^F_{\epsilon,s})\,,\\
			u = 0\,, &\text{ on } \Si(D^F_{\epsilon,s})\,,
		\end{cases}
	\end{equation}
	and get by Hölder's inequality
	\begin{multline}
		\label{eq:comp-pf} 
		\int_{\T(D^F_{\epsilon,s})} (u_\epsilon-v) \phi \dx \leq \int_{D^F_{\epsilon,s}} b (a-a_k) \Delta \psi_k \dxdt - \int_{D^F_{\epsilon,s}} f \psi_k \dxdt \\
		\leq \bigg [ \int_{D^F_{\epsilon,s}} b^2 \frac{(a-a_k)^2}{a_k} \dxdt \bigg ]^{1/2} \bigg [ \int_{D^F_{\epsilon,s}} a_k (\Delta \psi_k)^2 \dxdt \bigg ]^{1/2}\\
		+ \int_{D^F_{\epsilon,s}} f \psi_k \dxdt\,.
	\end{multline}

	To continue, we need to estimate the term on the right hand
        side of \eqref{eq:comp-pf} containing the quantity $a_k
        (\Delta \psi_k)^2$ independently of $k$.  To do this we follow
        the calculations of \cite{DK}. We use the equation
        \eqref{backwards_bvp} for $\psi_k$ and integrate by parts,
        first in time and then in space, which gives
	\begin{multline}\label{akDeltau_1}
		\int_{D^F_{\epsilon,s}} [a_k \Delta \psi_k] \Delta \psi_k \dxdt = - \int_{D^F_{\epsilon,s}} [\psi_k]_t \Delta \psi_k \dxdt \\
		= \int_{D^F_{\epsilon,s}} \psi_k [\Delta \psi_k]_t \dxdt - \int_{\T(D^F_{\epsilon,s})} \phi \Delta \phi \dx + \int_{\Bo(D^F_{\epsilon,s})} \psi_k \Delta \psi_k \dx \\
		= \int_{D^F_{\epsilon,s}} \psi_k \Delta [\psi_k]_t \dxdt + \int_{\T(D^F_{\epsilon,s})} |\grad \phi|^2 \dx - \int_{\Bo(D^F_{\epsilon,s})} |\grad \psi_k|^2 \dx \\
		\leq - \int_{\Si(D^F_{\epsilon,s})} \partial_n \psi_k [\psi_k]_t  \de \sigma \dt + \int_{D^F_{\epsilon,s}} \Delta \psi_k [\psi_k]_t \dxdt + \int_{\T(D^F_{\epsilon,s})} |\grad \phi|^2 \dx\,.
	\end{multline}
	Note now that the first term on the right hand side in \eqref{akDeltau_1} vanishes, since for almost every $t \leq s$ we have $[\psi_k]_t = 0$ on $\Si[D^F_{\epsilon,s}]$	due to the fact that $\psi_k$ vanishes smoothly on the boundary. 
	For the second term on the right hand side in \eqref{akDeltau_1}, we use the first line in \eqref{akDeltau_1}. This implies that 
	\begin{align} \label{akDeltau_2}
		\int_{D^F_{\epsilon,s}} a_k (\Delta \psi_k)^2 \dxdt \leq \frac{1}{2}\int_{\T(D^F_{\epsilon,s})}\abs{\nabla \phi}^2\dx\,. 
	\end{align}
	With the estimate \eqref{akDeltau_2} in hand, we see from \eqref{appconvergence} and the fact that $b$ is bounded that 
	\begin{equation}\label{conv1}
		\left ( \int_{D^F_{\epsilon,s}} b^2 \frac{(a-a_k)^2}{a_k} \dxdt \right )^{1/2} \left ( \int_{D^F_{\epsilon,s}} a_k (\Delta \psi_k)^2 \dxdt \right )^{1/2} \to 0 \quad \text{as} \quad k \to 0\,. 
	\end{equation}
	
	To proceed we need to take care of the term involving $f$ on the right hand side in \eqref{eq:comp-pf}. Recall that, as a distribution
	\begin{equation*}
		f = \frac{(1+\epsilon)^{m-1} - 1}{(1+\epsilon)^{m}} \Delta u^m\,.
	\end{equation*}
	Since the function $\psi_k$ vanishes on the lateral boundary $\Si(D^F_{\epsilon,s})$ of $D^F_{\epsilon,s}$, we have 
	\begin{multline} \label{fbound1}
		\int_{D^F_{\epsilon,s}} f \psi_k \dxdt = \frac{(1+\epsilon)^{m-1} - 1}{(1+\epsilon)^{m}} \int_{D^F_{\epsilon,s}} \grad u^m \spcdot \grad \psi_k \dxdt \\
		\leq \frac{(1+\epsilon)^{m-1} - 1}{(1+\epsilon)^{m}} \left (\int_{D^F_{\epsilon,s}} |\grad u^m|^2 \dxdt \right )^{1/2} \left (\int_{D^F_{\epsilon,s}} |\grad \psi_k|^2 \dxdt \right )^{1/2}. 
	\end{multline}
	By the assumption $u^m\in L^2(0,T;H^1_0(\Omega))$, we see that the first integral is bounded independent of $k$ and $\epsilon$. 
	
	Next we need estimate the $L^2$-norm of $|\grad \psi_k|$ independently of $k$, which we do as in \cite[p. 133]{VazquezBook}. Multiply the equation \eqref{backwards_bvp} for $\psi_k$ by the test function $\theta = \Delta \psi_k \chi(t)$, where $\chi(0) = 1/2$ and $\chi(s) = 1$; thus $\chi_t \approx \frac{1}{s}$. Next we integrate by parts, first in space and then in time, and get 
	\begin{multline}\label{vtrick}
          \begin{aligned}
            	0 = & \int_{D^F_{\epsilon,s}} [\psi_k]_t \Delta \psi_k \chi \dxdt + \int_{D^F_{\epsilon,s}} a_k (\Delta \psi_k)^2 \chi \dxdt \\
		=  &-\int_{D^F_{\epsilon,s}} \grad (\psi_k)_t \spcdot \grad \psi_k \chi \dxdt + \int_{D^F_{\epsilon,s}} a_k (\Delta \psi_k)^2 \chi \dxdt \\
		= & -\frac{1}{2}\int_{D^F_{\epsilon,s}} [|\grad \psi_k|^2]_t \chi \dxdt + \int_{D^F_{\epsilon,s}} a_k (\Delta \psi_k)^2 \chi \dxdt \\
		= &\frac{1}{2} \int_{D^F_{\epsilon,s}} (|\grad \psi_k|^2) \chi_t \dxdt - \frac{1}{2}\int_{\T(D^F_{\epsilon,s})} |\grad \psi_k|^2 \chi \\
		&+ \frac{1}{2}\int_{\Bo(D^F_{\epsilon,s})} |\grad \psi_k|^2 \chi + \int_{D^F_{\epsilon,s}} a_k (\Delta \psi_k)^2 \chi \dxdt\,,
          \end{aligned}	
	\end{multline}
	using \eqref{backwards_bvp} and \eqref{vtrick} we get
	\begin{multline}\label{psikbound}
		\frac{1}{s} \int_{D^F_{\epsilon,s}} (|\grad \psi_k|^2) \dxdt + \int_{D^F_{\epsilon,s}} a_k (\Delta \psi_k)^2 \chi \dxdt \\
		\leq C \left ( \int_{\T(D^F_{\epsilon,s})} |\grad \psi_k|^2 \chi \dx - \int_{\Bo(D^F_{\epsilon,s})} |\grad \psi_k|^2 \chi \dx \right ) 
		\leq C \int_{\T(D^F_{\epsilon,s})} |\grad \phi|^2 \dx\,.
	\end{multline}
	Combining \eqref{eq:comp-pf},\eqref{conv1},\eqref{fbound1} and \eqref{psikbound}, we have so far established
	\begin{equation}
		\label{eq:comp-pf2} \int_{\T(D^F_{\epsilon,s})} (u_\epsilon-v) \phi \dx\leq C\frac{(1+\epsilon)^{m-1}-1}{(1+\epsilon)^{m}}\int_{\T(D^F_{\epsilon,s})} |\grad \phi|^2 \dx\,.
	\end{equation}
	Before letting $\epsilon\to 0$, we still need to check that 
	\begin{equation*}
		\int_{\T(D^F_{\epsilon,s})} |\grad \phi|^2 \dx \leq C\,, 
	\end{equation*}
	for some constant $C$ not depending on $\epsilon > 0$. We are free to assume that $\phi \in C_0^\infty (D_{\epsilon_0,s})$ for some $\epsilon_0$. Then, since $D_{\epsilon_0,s} \subset D^F_{\epsilon,s}$ for $\epsilon<\epsilon_0$, we have $\T(D^F_{\epsilon,s}) \cap \overline {D_{\epsilon_0,s}} \subset \T(D_{\epsilon_0,s})$ which proves the desired bound. Thus, letting $\epsilon\to 0$ in \eqref{eq:comp-pf2}, we get that 
	\begin{align*}
		\int_{\overline {D_{\epsilon_0,s}} \cap [\R^n \times \{s\}]} (u-v) \phi \dx \leq 0\,.
	\end{align*}
	Since this holds for any positive $\phi$, we obtain that $u \leq v$ a.e. in $\Omega_T \cap [\R^n \times \{s\}]$ for any $s$, and then also in $\Omega_T$. 
\end{proof}

The crucial point in the proof above is that we can approximate the set
\begin{equation*}
	D_\epsilon = \left \{(x,t) \in \Omega_T: \frac{u}{1+\epsilon} \geq v \right \}\,,
\end{equation*}
by finite unions of space time boxes while staying inside the set where $u$ is a weak solution. Thus we can also deduce the following Theorem:

\begin{theorem}
	\label{Comparison2} Let $E$ be an open set in $\R^{n+1}$, let $u$ be a non-negative continuous weak solution in $E$ such that 
	\begin{equation} \nonumber \label{}
		\int_{E} [|u^m|^2+|\nabla u^m|^2] \dxdt < \infty\,.
	\end{equation}
	Let $v$ be a non-negative lower semicontinuous weak supersolution such that 
	\begin{equation} \nonumber \label{}
		\int_{E} [|v^m|^2+|\nabla v^m|^2] \dxdt < \infty\,,
	\end{equation}
	$v > 0$ on $\partial E$ and $u \leq v$ on $\partial E$. Then
        $u \leq v$ in $E$. Furthermore if a connected component of $\partial E$ is the boundary of a finite union of space-time cylinders then we can remove the assumption $v > 0$ on that component.
\end{theorem}

\section{The obstacle problem}
\label{sec:ost}

In this section, we construct solutions to the obstacle problem by a
potential theoretic method. More specifically, we call a function $u$
a solution to the obstacle problem if it is the smallest supersolution
above the given obstacle function $\psi$. 

Existence and uniqueness are fairly easily established for this notion
of solution to the obstacle problem. However, the relationship between
the variational solutions studied in \cite{BLS} and the smallest
supersolution is not immediately clear. In this direction, we apply
the comparison principle established earlier to prove that the
smallest supersolution is also a variational solution, provided that
the obstacle is sufficiently regular. This is a consequence of two
facts: first, we prove that the smallest supersolution is a point-wise
limit of variational solutions. Second, variational solutions are
stable with respect to convergence of the obstacles in a suitable
norm.

We expect that the converse is also true, i.e. that a variational
solution is the smallest supersolution . However, our version of the
comparison principle in general domains is not strong enough to prove
this.

First we describe the notion of smallest supersolution in more detail.
\begin{definition}
  Let $\psi$ be a positive, bounded measurable function in $\Omega_\infty$, and denote 
  \begin{displaymath}
    \U_{\psi}=\{v \text{ is a \lscss{} in }\Omega_\infty : v\geq \psi \text{ in }\Omega_\infty\}\,.
  \end{displaymath}
  We define the \emph{réduite} (or reduced function) of $\psi$ as
  \begin{displaymath}
    R_\psi=\inf\{v:v\in \mathcal{U}_{\psi}\}\,.
  \end{displaymath}
  For a measurable set $E$, we abbreviate $R_E=R_{\chi_E}$. We denote
  by $\widehat{R}_\psi$ (lower semicontinuous)
  $\essliminf$-regularization of $R_\psi$. The function
  $\widehat{R}_\psi$ is usually called the \emph{balayage} of $\psi$.
  
\end{definition}
The terms r\'eduite and balayage come from classical potential
theory. The notion is due to Poincaré.  We will need the following
basic theorem, for which the proof is standard, but we reproduce it
here for the reader's convenience.
\begin{theorem}\label{thm:bala-bala}
  The balayage $\widehat{R}_\psi$ is a \lscss{} in $\Omega_T$.
\end{theorem}
\begin{proof}
  Pick a space-time cylinder $U_{t_1,t_2}\Subset \Omega_T$ and a weak
  solution $u$ which is continuous in $\overline{U}_{t_1,t_2}$ with
  $u\leq \widehat{R}_\psi$ on $\partial_p U_{t_1,t_2}$. Then also
  $u\leq v$ on $\partial_p U_{t_1,t_2}$ for $v\in \U_{\psi}$, and by
  comparison the same holds in $U_{t_1,t_2}$. We take the infimum over
  $v$ to get that $u\leq R_\psi$.  Since $\widehat{u}=u$ by the
  continuity of $u$, we conclude that $u\leq \widehat{R}_\psi$.
\end{proof}

\emph{Note that in general, $R_\psi$ might not be lower
  semicontinuous, and $\widehat{R}_\psi$ might not be above $\psi$ in
  every point.}  However, for \emph{continuous} $\psi$ it holds that $\widehat{R}_\psi\geq \psi$ \emph{everywhere}. This together with
Theorem \ref{thm:bala-bala} implies that $ \widehat{R}_\psi$ is the
unique smallest \lscss{} above the obstacle $\psi$. 
By the smallest
supersolution, we mean a function $u\in \mathcal{U}_{\psi}$ with the
property that
\begin{equation}
	\label{eq:smallest} u\leq v \quad \text{for all} \quad v\in \mathcal{U}_{\psi}\,. 
\end{equation}
A \lscss{} with the property \eqref{eq:smallest} is unique, if it
exists; indeed, if there are two functions $u_1,u_2\in
\mathcal{U}_{\psi}$ satisfying \eqref{eq:smallest}, then two
applications of \eqref{eq:smallest} give the inequalities $u_1\leq
u_2$ and $u_2\leq u_1$, so that $u_1=u_2$. 

The next aim is to relate the smallest supersolution to the
variational solutions to the obstacle problem constructed in
\cite{BLS}.  We first recall some facts from \cite{BLS}.

We consider nonnegative obstacle functions $\psi$ defined on
$\Omega_T$, with compact support and satisfying
\begin{equation}
	\label{obstacle} \psi^m\in L^2\big(0,T;H^{1}_0(\Omega)\big) ,\quad 
	\partial_t(\psi^m)\in L^{\frac{m+1}{m}}(\Omega_T)\,. 
\end{equation}
The class of admissible functions for the obstacle problem is defined by 
\begin{align*}
	K_{\psi}(\Omega_T) := \big\{v\,:\,\Omega_T\to [0,\infty]\,:\, v^m\in L^2(0,T;H^{1}_0(\Omega))\,, v\ge\psi\mbox{ a.e. on $\Omega_T$}\big\}\,. 
\end{align*}
Note that $\psi\in K_{\psi}$, and therefore $K_{\psi} \neq \emptyset$.

With the above classes, we can state the definition of a strong
solution to the obstacle problem.
\begin{definition}
  \label{def:strong} A nonnegative function $u\in
  K_{\psi}(\Omega_T)$ is a {\it strong solution to the obstacle
    problem for the porous medium equation} if $
	\partial_t u\in L^2(0,T;H^{-1}(\Omega))$ and 
	\begin{equation*} 
		  \int_0^T\inprod{\partial_t u,\alpha(v^m-u^m)}\dt+ \int_{\Omega_T}\alpha \nabla u^m\spcdot \nabla(v^m-u^m)\de z \geq 0\,,
	\end{equation*}
	holds for all comparison maps $v\in K_{\psi}(\Omega_T)$ and every Lipschitz continuous cut-off function $\alpha:[0,T]\to [0,\infty]$ with $\alpha(T)=0$.
\end{definition}
The cutoff function $\alpha$ is needed for making this definition
consistent with the definition of weak solutions to the obstacle
problem, which we will recall later.

For the existence of strong solutions, we still need the assumption 
\begin{equation}
	\label{Psi} \Psi:=
	\partial_t\psi-\Delta \psi^m\in L^\infty(\Omega_T)\,. 
\end{equation}

The following result can be extracted from \cite[Theorem~2.6]{BLS}.
\begin{theorem}
  \label{thm:pme-obstacle} Let $\Omega$ be a bounded open subset of
  $\R^n$ with a smooth boundary.   Assume that the obstacle
  $\psi$ satisfies the regularity conditions \eqref{obstacle} and
  \eqref{Psi}. Then there exists a strong solution $u$ to the obstacle
  problem for the PME in the sense of Definition
  \ref{def:strong} satisfying $u^m\in L^2(0,T;H^1_0(\Omega))$ and
  $u(\spcdot,0)=0$.
	
  The function $u$ is also locally H\"older continuous, and satisfies
  $u\geq \psi$ everywhere in $\Omega_T$. Further, $u$ is a weak
  supersolution to the porous medium equation in $\Omega_T$, and a
  weak solution in the open set $\{z\in \Omega_T:u(z)>\psi(z)\}$.
\end{theorem}

We now wish to show that $u$ in Theorem \ref{thm:pme-obstacle} is a weak solution in the larger set
\begin{equation} \nonumber \label{}
	[\Omega_T \setminus \supp(\psi)]\cup\{z\in \Omega_T:u(z)>\psi(z)\}\,.
\end{equation}
With this in mind we recall the following form of a partition of unity.

\begin{lemma}
  [Partition of Unity]\label{partition} Let $U_1,U_2,\ldots,U_n$ be
  open sets, and let $K$ be a compact set such that $K\subset U_1\cup
  U_2\cup\cdots\cup U_n$. Then there exist functions $\eta_i\in
  C_0^\infty(U_i)$ such that
  \begin{displaymath}
    \sum_{i=1}^n \eta_i=1 \quad \text{on}\quad K\,. 
  \end{displaymath}
\end{lemma}
\begin{proof}
  For a version where the functions $\eta_i$ are continuous, see
  \cite[Theorem 2.13, p. 40]{RCA}. The fact that one may also choose
  smooth functions follows easily from the continuous version by
  applying a suitable mollification.
\end{proof}
\begin{lemma}
  \label{SuppLem} The strong solution to the obstacle problem given by
  Theorem~\ref{thm:pme-obstacle} is also a weak solution to the PME in
  the set $\Omega_T\setminus \supp(\psi)$.
\end{lemma}
\begin{proof}
  Let $\delta>0$ be a number, and let $\eta_\delta:\R\to [0,1]$ be a
  Lipschitz function with $\eta(s)=0$ for $s\leq -\delta$, $\eta(s)=1$
  for $s\geq 0$, and $\abs{\eta'(s)}\leq 1/\delta$. The solution $u$
  is constructed in \cite{BLS} as the uniform limit as $\delta\to 0$
  of solutions to
  \begin{displaymath}
    \partial_t u_\delta -\Delta u^m_\delta=\eta_\delta(\psi^m-u_\delta^m) (
    \partial_t \psi-\Delta \psi^m)_+ \,. 
  \end{displaymath}
  The claim now follows from the fact that $(
  \partial_t \psi-\Delta \psi^m)_+=0$ in $\Omega_T\setminus \supp(\psi)$. 
\end{proof}

\begin{theorem}
  \label{lem:obstacle} Let $\psi$ be a nonnegative, compactly
  supported function satisfying the regularity assumptions
  \eqref{obstacle} and \eqref{Psi}, let $u$ be the strong solution to the obstacle problem given by
  Theorem~\ref{thm:pme-obstacle} with obstacle $\psi$, and denote
  $K=\supp(\psi)\cap\{u=\psi\}$. Then $u$ is a weak solution in
  $\Omega_\infty\setminus K$.
\end{theorem}
\begin{proof}
  We first show that the function $u$ is a weak solution in
  $\Omega_\infty \setminus K$. Denote $U_1=\Omega_\infty
  \setminus\supp(\psi)$ and $U_2=\{u>\psi\}$. These sets are open, and
	\begin{displaymath}
		\Omega_\infty \setminus K=U_1\cup U_2 \,. 
	\end{displaymath}
	Further, $u$ is a weak solution in $U_1$ and in $U_2$. The
        claim concerning the set $U_1$ is Lemma \ref{SuppLem}, and the
        claim about $U_2$ is a part of Theorem
        \ref{thm:pme-obstacle}. To show that $u$ is a solution also in
        $U_1\cup U_2$, let $\varphi\in C_0^\infty(U_1\cup U_2)$. An
        application of Lemma \ref{partition} shows that there are
        functions $\eta_i\in C_0^\infty(U_i)$, $i=1,2$, such that
        $\eta_1+\eta_2=1$ on the support of $\varphi$. By applying the
        fact that $u$ is a weak solution in $U_1$ and in $U_2$, we get
	\begin{displaymath}
		\int_{\Omega_\infty}-u
		\partial_t \varphi+\nabla u^m\spcdot \nabla\varphi\dxdt= \sum_{i=1}^2\int_{\Omega_\infty}-u
		\partial_t(\varphi\eta_i) +\nabla u^m\spcdot \nabla(\varphi\eta_i)\dxdt=0 \,.
	\end{displaymath}
	Since this holds for any test function $\varphi$, $u$ is a
        weak solution in $U_1\cup U_2$.
\end{proof}	

We are now ready to proceed with the approximation result.
\begin{theorem}\label{thm:approxpox}
  Let $\psi$ be continuous and compactly supported in $\Omega_T$.
  Then the smallest supersolution $\widehat{R}_{\psi}$ is an
  increasing limit of strong solutions $w_j$ to the obstacle problem
  with smooth compactly supported obstacles $\phi_j$ increasing to
  $\psi$.
\end{theorem}
\begin{proof}
	Let $U_j=\{\psi>1/j\}$ for $j=1,2,\ldots,$ $K_j=\overline{U}_j$. 
	First note that $K_j$ and $K_{j+1}$ have a positive distance
        between them. Thus if we let $h_j = (\sqrt{\psi}-1/\sqrt{j})_+$, we see that $h_{j+1}-h_j$ is strictly positive in $K_{j}$. By a mollification argument it can easily be seen that for each $j=1,\ldots$ there exists a function $f_j \in C_0^\infty(K_{j+1})$ such that
	\begin{equation} \nonumber \label{}
		h_j \leq f_j \leq h_{j+1}\,.
	\end{equation}
	Taking $\phi_j = f_j^2$, $j=1,\ldots$, we immediately see that $\phi_j^m \in C_0^2(K_{j+1})$, thus it satisfies \eqref{obstacle} and \eqref{Psi}. Moreover by construction we get
	\begin{equation} \nonumber \label{}
		\phi_1 < \phi_2 < \ldots < \psi, \quad \phi_j \to \psi \text{ as } j \to \infty\,.
	\end{equation}
	Let $w_j$ be the strong solutions to the $\phi_j$-obstacle
        problems. Since $\widehat{R}_\psi\geq \psi$ by the continuity
        of $\psi$, we have $w_j<\widehat{R}_\psi$ on
        $K=\partial(\{w_j=\phi_j\}\cap \supp\phi_j)$. Also note that
        $K \subset U_{j+1}$, whence $\widehat{R}_{\psi} >
        \frac{1}{j+1} > 0$ on $K$.  This allows us to use the
        comparison principle of Theorem \ref{Comparison1} together
        with Theorem \ref{lem:obstacle} to get that
        $w_j\leq\widehat{R}_\psi$. A similar argument shows that
        $w_j\leq w_{j+1}$. Thus $w=\lim_{j\to \infty} w_j$ is a
        \lscss{} as an increasing limit of continuous supersolutions,
        and $w\leq \widehat{R}_\psi$. To finish the proof, we have
        that $w\geq \psi$ everywhere in $\Omega_T$, whence $R_\psi\leq
        w$. Thus
	\begin{equation*}
		w\leq \widehat{R}_\psi\leq R_\psi \leq w\,,
	\end{equation*}
	and the proof is complete. 
\end{proof}

The final step is to combine the approximation result with a stability
result for variational solutions to conclude that the smallest
supersolution is also a variational solution.  We recall some more
facts from \cite{BLS}, in particular the notion of a weak variational
solution, for which stability with respect to the obstacles can be
established.

For the notion of weak solutions, we use the class of admissible
comparison functions 
\begin{equation} \nonumber \label{}
	K_{\psi}'(\Omega_T) = \big\{v\in K_{\psi}(\Omega_T): \partial_t(v^m)\in L^{\frac{m+1}{m}}(\Omega_T)\big\}\,.	
\end{equation}
We need to make sense of the time term in the variational inequality
when we do not know that $\partial_t u$ belongs to the dual of the
parabolic Sobolev space. We do this as in \cite{AltLuckhaus} and
\cite{BLS}. We recall the notation
\begin{align*}
  \dinprod{
    \partial_t u,\alpha\eta(v^m-u^m)}_{u_0}=&\int_{\Omega_T}\eta\left[\alpha'\left[\frac{1}{m+1}u^{m+1}-u v^m\right]-\alpha u
    \partial_t v^m\right]\dxdt\\
	& +\alpha(0)\int_{\Omega}\eta\left[\frac{1}{m+1}u^{m+1}_0-u_0v^m(\spcdot,0)\right]\dx 
\end{align*}
where $u_0\in L^{m+1}(\Omega)$ is a function giving the initial values
of the solution, and $\alpha$ is a nonnegative Lipschitz continuous
cutoff function depending only on the time variable with
$\alpha(T)=0$. The role of the function $\alpha$ is to eliminate the
final time term, as we do not know in general whether $u$ is
continuous in time. Observe that if $
\partial_t u\in L^2(0,T;H^{-1}(\Omega))$, we have 
\begin{displaymath}
	\int_0^T\inprod{
	\partial_t u,\alpha\eta(v^m-u^m)}\dt=\dinprod{
	\partial_t u,\alpha\eta(v^m-u^m)}_{u_0}\,. 
\end{displaymath}
This follows formally from integration by parts, and the rigorous
justification is given in \cite[Lemma~3.2]{BLS}. This makes the
following definition consistent with the definition of strong
solutions in the previous section, i.e. strong solutions are also weak
solutions.
\begin{definition}
  \label{def:loc-weak} A nonnegative function $u\in
  K_{\psi}(\Omega_T)$ is a {\it weak solution to the obstacle problem
    for the porous medium equation} if the inequality
	\begin{equation*}
          \dinprod{
            \partial_tu,\alpha\eta(v^m-u^m)}_{u_0} + \int_{\Omega_T}\alpha \nabla u^m\spcdot \nabla\big(\eta(v^m-u^m)\big)\de z \ge 0 
	\end{equation*}
	holds true for all comparison maps $v\in K'_{\psi}(\Omega_T)$ and every nonnegative, Lipschitz continuous cut-off function depending only on the time variable with $\alpha(T)=0$.
\end{definition}

\begin{theorem} \label{thm:pme-obstacle-weak} 
  Let $\Omega$ be a bounded open subset of
  $\R^n$ with a smooth boundary. Assume that the obstacle $\psi$
  satisfies the regularity condition \eqref{obstacle}. Then there
  exists a weak solution $u$ to the obstacle problem for the porous
  medium equation in the sense of Definition \ref{def:loc-weak}
  satisfying $u^m\in L^2(0,T;H^1_0(\Omega))$.
	%
	Again, $u$ is a weak supersolution to the porous medium equation in $\Omega_T$.
\end{theorem}

The following theorem may be extracted from the proof of Theorem~2.7
in \cite{BLS}.
\begin{theorem}\label{thm:stabab}
  Let $\psi_i$ be a sequence of obstacles satisfying \eqref{obstacle} with compact support in
  $\Omega_T$ such that
  \begin{equation*}
	  \psi_i^m\to \psi^m\text{ in } L^2(0,T;H^1_0(\Omega))\,, \quad\text{and}\quad
	  \partial_t (\psi_i^m)\to \partial_t(\psi^m) \text{ in }L^{\frac{m+1}{m}}(\Omega_T)\,,
  \end{equation*}
  furthermore let $u_i$ be the respective variational weak solutions to the
  obstacle problem with obstacle $\psi_i$, see Theorem \ref{thm:pme-obstacle-weak}.

  Then there is a function $u\in L^\infty(0,T;L^{m+1}(\Omega))$ with
  $u^m\in L^2(0,T;H^1_0(\Omega))$ and $u(\spcdot, 0)=0 $ such
  that, up to subsequences,
  \begin{equation*}
    u_i\to u \text{ a.e.},\quad u^m_i\to u^m \text{ in
    }L^2(\Omega_T), \quad \text{and }\nabla u^m_i\to \nabla u^m\text{
      weakly in }L^{2}(\Omega_T)\,. 
  \end{equation*}
  Furthermore, $u$ is a variational weak solution to the obstacle problem
  with obstacle $\psi$ and initial values zero.
\end{theorem}

Since strong variational solutions are also weak variational
solutions, we get the following theorem as an immediate consequence of
Theorems~\ref{thm:approxpox}~and~\ref{thm:stabab}.

\begin{theorem}\label{thm:smallest-is-variational}
  Let $\psi$ be a continuous function with compact support in
  $\Omega_T$ satisfying the regularity assumptions
  \eqref{obstacle}. Then the smallest supersolution
  $\widehat{R}_\psi$ is also a variational weak solution.
\end{theorem}

A general converse for Theorem~\ref{thm:smallest-is-variational}
remains open. We record the following partial result for use in
Section~\ref{sec:cap}.

\begin{theorem}\label{thm:variational-is-smallest-crappyversion} 
  Let $\psi$ be a smooth obstacle with
  $\psi>0$ in $\Omega_T$. Then any variational strong solution $u$ to
  the obstacle problem coming from Theorem~\ref{thm:pme-obstacle} 
  satisfies $u=\widehat{R}_\psi$.
\end{theorem}
\begin{proof}
  Since $\psi>0$, any strong solution is strictly
  positive inside $\Omega_T$. Thus, given a semicontinuous supersolution $v\in \U_\psi$,
  we may apply Theorem \ref{Comparison1} on the set $\{u>\psi\}$ to
  conclude that $u\leq v$. Since $u\in \U_\psi$, we get
  $u=\widehat{R}_\psi$.
\end{proof}

\section{Parabolic capacity for the porous medium equation}
\label{sec:cap}

In this section, we define the parabolic capacity for the porous medium equation and establish its basic properties. 
\begin{definition}
	The PME capacity of an arbitrary subset $E$ of $\Omega_{\infty}$ is 
	\begin{equation*}
		\capacity(E) = \sup \{\mu(\Omega_\infty): 0 \leq u_{\mu} \leq 1, \supp(\mu) \subset E \}\,, 
	\end{equation*}
	where $\mu$ is a positive Radon measure, and $u_\mu$ is a weak
        supersolution with $u_{\mu} = 0$ on $\partial_p
        \Omega_\infty$, and a weak solution to the measure data problem
        \begin{displaymath}
          (u_{\mu})_t-\Delta u^m_{\mu}=\mu\,.
        \end{displaymath}
\end{definition}

Our next result is that there exists a \emph{capacitary extremal} for
the PME capacity of a compact set $K$, i.e. a \lscss{} $u$ such that
$\capacity(K)=\mu_u(K)$. We need the following two lemmas.
\begin{lemma}\label{lem:Comparison15}
  Let $\psi$ be a smooth, positive compactly supported function, and
  set
  \begin{equation*}
    \psi_\varepsilon=(\psi^m+\varepsilon^m)^{1/m}\quad \text{and}\quad v_\varepsilon=\widehat{R}_{\psi_\varepsilon}\,.
  \end{equation*}
  Then the limit function
  \begin{equation*}
    v=\lim_{\varepsilon\to 0}v_\varepsilon
  \end{equation*}
  is a continuous weak supersolution and a weak variational solution
  to the obstacle problem with obstacle $\psi$ in $\Omega_T$. Further,
  $v$ is a weak solution in the open set $\{v>\psi\}$.
\end{lemma}
\begin{proof}
  The existence of the point-wise limit as $\varepsilon\to 0$ follows from the
  fact that $\widehat{R}_{\psi_{\varepsilon_1}}\leq \widehat{R}_{\psi_{\varepsilon_2}}$
  if $\varepsilon_1\leq \varepsilon_2$.  The limit $v$ is an upper semicontinuous
  weak supersolution as a decreasing limit of continuous weak
  supersolutions, and $v\geq \psi$ since $v_\varepsilon\geq \psi_\varepsilon$.  By
  Theorem~\ref{thm:variational-is-smallest-crappyversion}, we may take
  $v_\varepsilon$ to be a strong variational solution to the obstacle
  problem.  Hence $v$ is a weak variational solution to the obstacle
  problem by \cite{BLS}. The continuity follows from \cite{BLS2}.

  Since $v_\varepsilon$ is a variational strong solution to the obstacle
  problem, it is weak solution in the set $\{v_\varepsilon>\psi_\varepsilon\}$. If
  $K$ is now a compact set contained in $\{v>\psi\}$, we have that $K$
  is also contained in $\{v_\varepsilon>\psi_\varepsilon\}$ for all sufficiently
  small $\varepsilon$, since 
  \begin{equation*}
    v_\varepsilon-\psi_\varepsilon\geq \inf_{K}(v-\psi)-\varepsilon\,
  \end{equation*}
  by the inequalities $v_\varepsilon\geq v$,
  $-\psi_\varepsilon=-(\psi^m+\varepsilon^m)^{1/m}\geq -\psi-\varepsilon$,
  and the fact that $\inf_{K}(v-\psi)>0$. Thus 
  \begin{equation*}
    \int_{\Omega_T}-v\partial_t \varphi+\nabla v^m\spcdot \nabla
    \varphi\dx\dt=\lim_{\varepsilon\to 0}\int_{\Omega_T}-v_\varepsilon\partial_t \varphi+\nabla v^m_\varepsilon\spcdot \nabla
    \varphi\dx\dt=0
  \end{equation*}
  for all smooth test functions $\varphi$ with support in $K$. Since
  $K$ was arbitrary, $v$ is a weak solution.
\end{proof}

The next lemma is the key step in constructing the capacitary
extremal. For the proof, we record the following estimate. Let $u$ be
a positive weak supersolution in $\Omega_\infty$, $u$ vanishing on the
lateral boundary of $\Omega_\infty$. Suppose in addition that there
exists a time $t_0\geq 0$ so that $u$ is a weak solution in $\Omega_{t_0,\infty}$. 
Then 
\begin{equation}\label{universal}
  u(x,t)\leq c(t-t_0)^{-\frac{1}{m-1}}
\end{equation}
for all $t>t_0$ with a constant depending only on $n$, $m$, and the
diameter of $\Omega$.  This is the so-called universal estimate. See
Proposition~5.17 in \cite{VazquezBook} for the proof.

\begin{lemma}
  \label{Comparison3} Let $K$ be a compact subset of
  $\Omega_\infty$. Assume that $u$ and $v$ are lower semicontinuous
  weak supersolutions in $\Omega_\infty$ and that u is continuous in
  $\overline \Omega_\infty$, and a weak solution after a time $T$ such
  that $K\Subset \Omega_T$. Moreover, assume that $u > 1$ in $K$, $u =
  0$ on $
  \partial_p \Omega_\infty$, $0 \leq v \leq 1$ in $\Omega_\infty$, and $v=0$ on $
  \partial_p\Omega_\infty$. Then 
  \begin{equation*}
    \mu_v(K) \leq \mu_{u} (\Omega_\infty)\,. 
  \end{equation*}
\end{lemma}
\begin{proof}
  Let $0 \leq \psi_i \in C^\infty_0(\Omega_\infty)$ be an increasing
  sequence of smooth obstacles converging to $v$ such that $\psi_i <
  v$ and $\psi_i<\psi_{i+1}$. Denote the perturbed obstacles
  $\psi_i^\epsilon = (\psi_i^m + \epsilon^m)^{1/m}$, and the
  corresponding solutions to the obstacle problem by $v_i^\epsilon$.
  Further, let $v_i=\lim_{\epsilon\to 0}v_i^\epsilon$ be the weak
  supersolutions in $\Omega_\infty$ constructed in 
  Lemma~\ref{lem:Comparison15}. We argue as in \cite[proof of Theorem
  3.2, p. 148--149]{KinnunenLindqvist2} to see that $v_{i}\leq
  v_{i+1}\leq v$ and $v_i\to v$ as $i\to \infty$. Finally, we have that
  \begin{equation*}
    \sup_{\Omega_\infty} v_i^\epsilon = \sup_{\Omega_\infty} \psi_i^\epsilon \leq 1+\epsilon
  \end{equation*}
  and
    \begin{equation*}
      \sup_{\Omega_\infty} v_i = \sup_{\Omega_\infty} \psi_i \leq 1\,.
  \end{equation*}
  We define the supersolution 
  \begin{equation*}
    w_i^\epsilon = \min(v_i^\epsilon,u)\,. 
  \end{equation*}
  By lower semicontinuity, the set $\{u > 1\}$ is open, and $K$ is
  compactly contained in it. This allows us to construct a compact set
  $K'$ and an open set $U$ such that $K \subset U \subset K'\subset
  \{u>1\}$. If $\epsilon$ is small enough we know that $1+\epsilon <
  u$ in $K'$, so that $w_i^\epsilon=v_i^\epsilon$ in $K'$. Hence for
  such $\epsilon$ we have for $\phi' = 1$ on $U$ and $\phi' = 0$
  outside $K'$ that
  \begin{equation}
    \label{MuLowerBound} \mu_{v_i^\epsilon}(U) \leq \int_U \phi' \de \mu_{v_i^\epsilon} = \int_U \phi' \de \mu_{w_i^\epsilon}\leq \mu_{w_i^\epsilon}(K')\,. 
  \end{equation}
  
  Next note that since $u = 0$ on $
  \partial_p \Omega$ and $u\leq \epsilon$ for sufficiently large times
  by \eqref{universal}, there is a compact set $K'' \supset K'$ in
  $\Omega_\infty$ such that in $\Omega_\infty \setminus K''$ we have
  $w_i^\epsilon = u$ since $v_i^\epsilon\geq \epsilon$.  . Hence we
  obtain for $\phi'' \in C_0^\infty(\Omega_\infty)$ such that $\phi'' = 1$ on $K''$ that
  \begin{align*}
    \int_{\Omega_\infty} \phi'' \de\mu_{w_i^\epsilon}=&\int_{\Omega_\infty} -w_i^\epsilon\frac{
      \partial \phi''}{
      \partial t} +\nabla (w_i^\epsilon)^m\spcdot\nabla\phi''\dx\dt\\
    =& \int_{\Omega_\infty}-u^\epsilon\frac{
      \partial \phi''}{
      \partial t} +\nabla u^m\spcdot\nabla\phi''\dx\dt\\=&\int_{\Omega_\infty}\phi''\de\mu_{u}\,. 
  \end{align*}
  Thus we obtain the estimate 
  \begin{equation}
    \label{MuUpperBound} \mu_{w_i^\epsilon}(K')\leq \int_{\Omega_\infty} \phi'' \de \mu_{w_i^\epsilon} = \int_{\Omega_\infty} \phi'' \de \mu_{u} \leq \mu_{u}(\Omega_\infty)\,. 
  \end{equation}
  We combine \eqref{MuLowerBound} and \eqref{MuUpperBound} to get the inequality 
  \begin{equation*}
    \mu_{v_i^\epsilon}(U) \leq \mu_{u}(\Omega_\infty)\,. 
  \end{equation*}
  By construction $v_i^\epsilon\to v_i$ point-wise, thus from Lemma \ref{Convergence} we get that $\mu_{v_i^\epsilon} \to \mu_{v_i}$ weakly.
  By the standard properties of weak convergence of
  measures, see Theorem \ref{thm:weak-meas}, we get that
  \begin{equation*}
    \mu_{v_i}(U) \leq \liminf_{\epsilon \to 0} \mu_{v_i^\epsilon}(U) \leq \mu_{u}(\Omega_\infty)\,. 
  \end{equation*}
  The sequence $(v_i)$ is increasing, and converges point-wise to the
  original supersolution~$v$. Again from Lemma \ref{Convergence} we get the weak convergence of the
  corresponding measures. Another application of Theorem
  \ref{thm:weak-meas} now shows that
	\begin{equation*}
          \mu_{v} (K) \leq \mu_{v}(U) \leq \liminf_{i \to \infty} \mu_{v_i}(U) \leq \mu_{u}(\Omega_\infty)\,, 
	\end{equation*}
	and the proof is complete.
\end{proof}

A consequence of Theorem \ref{Comparison1}, is that in the special
case that we have a decreasing sequence of smooth obstacles converging
to a characteristic function of a compact set, the obstacle problem is
stable. If we had a full elliptic comparison principle, this lemma
would hold for a decreasing sequence of smooth obstacles converging to
an upper semi-continuous obstacle.

\begin{lemma} \label{Stability} 
	Let $K \subset \Omega_\infty$ be a compact set. Let $E_i \Subset \Omega_\infty$, $i=1,\ldots$ be a shrinking sequence of open sets such that $E_{i+1} \Subset E_{i}$
	\begin{equation*}
		\bigcap_{i=1}^{\infty} \overline E_i = K\,.
	\end{equation*}
	Assume that the non-negative functions $\psi_i:\Omega_\infty
        \to \R$ are supported in $\overline E_i$, satisfy
        \eqref{obstacle} and \eqref{Psi}, and $\psi_i \geq \chi_K$, $i
        = 1,2,\ldots,$ is a decreasing sequence such that $\psi_i \to
        \chi_K$ point-wise in $\Omega_\infty$ as $i \to \infty$. Then
        $R_{\psi_i} \to R_{K}$ point-wise in $\Omega_\infty$ and
        $\mu_{R_{\psi_i}} \to \mu_{R_K}$ weakly as $i \to \infty$.
\end{lemma}
\begin{proof}
  By Theorem \ref{lem:obstacle}, the functions $R_{\psi_i}$ are
  continuous. Thus an application of Lemma \ref{Convergence} shows
  that $u = \lim_{i\to \infty} R_{\psi_i}$ is an upper semicontinuous
  weak super-solution, and the respective measures also converge
  weakly. Further,
	\begin{equation*}
		u \geq R_{K}\,, 
	\end{equation*}
	since $R_{\psi_i}\geq R_K$ for each $i$. 
	
	The lemma now follows if we prove the opposite inequality. To
        this end, note first and that from Theorem \ref{lem:obstacle},
        $R_{\psi_i}$ is a weak solution in $\{R_{\psi_i}>\psi_i\}\cup
        (\Omega_\infty\setminus\supp(\psi_i))$, so that the support of
        the measure $\mu_{R_{\psi_i}}$ is contained in $\supp(\psi_i)
        \subset \overline E_i$. These sets shrink to $K$, and the
        measures $\mu_{R_{\psi_i}}$ converge weakly to $\mu_u$. Thus
        $\supp(\mu_u) \subset K$, which implies that $u$ is a weak
        solution in $\Omega_\infty\setminus K$.

	If now $v \geq \chi_K$ is an arbitrary \lscss\ with $v=0$ on $
	\partial_p \Omega_\infty$, it follows from Theorem \ref{Comparison1} that $u \leq v$. We take the infimum over $v$ to get that 
	\begin{equation*}
		u \leq R_{K}\,, 
	\end{equation*}
	and the proof is complete. 
\end{proof}
A consequence of the stability Lemma \ref{Stability} is that we have
stability of the \emph{balayage} with respect to decreasing sequences of compact sets.
\begin{lemma}\label{Stability2}
	Let $K_i \subset \Omega_\infty$, $i=1,2,\ldots$, be a decreasing sequence of compact sets and denote $K = \cap_{i=1}^\infty K_i$. Then $\hat R_{K_i}$ is a decreasing sequence converging to $\hat R_{K}$, moreover $\mu_{\hat R_{K_i}}$ converges to $\mu_{\hat R_K}$, weakly as $i \to \infty$.
\end{lemma}
\begin{proof}
	Let us construct $E_i = \{d((x,t);K_i) < c/i\}$, $i=1,\ldots$, for a small constant $c < 1$ such that $\overline E_1 \subset \Omega_\infty$, then the sequence $E_i$ satisfies the requirements of Lemma \ref{Stability}. 
	
	Let us now construct smooth functions $\hat \psi_i \in C_0^\infty(\overline E_i)$ such that $\hat \psi_i = 1$ on $K_i$, then let $\psi_i = [\hat \psi]^2$, and we have that $R_{\psi_i} \geq R_{K_i}$ by construction. As in the proof of Theorem \ref{thm:approxpox}, the sequence $\psi_i$ will satisfy \eqref{obstacle} and \eqref{Psi}. It is now clear that the sequence $\psi_i$ satisfies all requirements of Lemma \ref{Stability} and thus we get that $R_{\psi_i} \to R_K$ and consequently also $R_{K_i} \to R_K$, furthermore using Lemma \ref{Convergence} we see that the measures $\mu_{R_{K_i}}$ converge weakly to $\mu_{R_K}$ as $i \to \infty$. 
\end{proof}

\begin{theorem}
	\label{thm:extremal} Let $K$ be a compact subset of $\Omega_\infty$. Then 
	\begin{equation*}
		\capacity(K) = \mu_{\hat R_{K}}(K)\,. 
	\end{equation*}
\end{theorem}
\begin{proof}
	Since $\hat R_K$ is a \lscss\ such that $0 \leq \hat R_K \leq 1$, it follows immediately from the definition of the PME capacity that 
	\begin{equation*}
		\mu_{\hat R_K} (K) \leq \capacity(K)\,, 
	\end{equation*}
	since $\hat R_K$ is a solution outside $K$.
	
	To prove the opposite inequality, let first $K' \Subset \Omega_\infty$ be a compact set such that $K \Subset K'$. To be able to use Lemma \ref{Stability} we will let $E_i \subset K'$, $i=1,\ldots$ be a shrinking sequence of open sets such that 
		\begin{equation} \nonumber \label{}
			\bigcap_{i=1}^{\infty} \overline E_i = K\,.
		\end{equation}
	Let $\hat \psi_i \in C_0^{\infty}(\overline E_i)$, $i=1,\ldots$, be a decreasing sequence of smooth functions converging to $\chi_K$ point-wise in $\Omega_\infty$ as $i \to \infty$, and such that 
	\begin{equation*}
		\hat \psi_i = \sqrt{1+\frac{1}{2^i}} \quad \text{ on $K$}\,. 
	\end{equation*}
	Consider now the functions $\psi_i = [\hat \psi_i]^2$, then $\psi_i^m \in C_0^2(\overline E_i)$, and it is a decreasing sequence of functions converging to $\chi_K$ point-wise in $\Omega_\infty$ as $i \to \infty$, such that 
	\begin{equation*}
		\psi_i = 1+\frac{1}{2^i} \quad \text{ on $K$}\,,
	\end{equation*}
	moreover $\psi_i$ satisfies \eqref{obstacle} and \eqref{Psi} for all $m > 1$.
	Denote by $u_i$ the corresponding solutions to the obstacle problems with obstacle $\psi_i$. Let now $v$ be a weak supersolution in $\Omega_\infty$ such that $0 \leq v \leq 1$ and $v = 0$ on $
	\partial_p \Omega_\infty$. Then it follows from Lemma \ref{Comparison3} that 
	\begin{equation*}
		\mu_v(K) \leq \mu_{u_i}(\Omega_\infty) = \mu_{u_i}(K')\,. 
	\end{equation*}
	We use Lemma \ref{Stability} to see that $\mu_{u_i} \to \mu_{\hat R_K}$ weakly. 
	The claim now follows from the above estimate, since 
	\begin{equation*}
		\lim\sup_{i \to \infty} \mu_{u_i}(K') \leq \mu_{\hat R_K} (K') = \mu_{\hat R_K} (K) 
	\end{equation*}
	by Theorem \ref{thm:weak-meas}. 
\end{proof}

We have now developed all the technical tools needed to establish the basic properties of the PME capacity, including that it is a regular, subadditive capacity.

\begin{theorem} \label{thm_prop}
	The PME capacity has the following properties.
	\begin{enumerate}
		\item \label{subadd} {\em Countable subadditivity:} In other words if $E_i$, $i=1,2,\ldots,$ be arbitrary subsets of $\Omega_\infty$ and $E = \cup_{i=1}^\infty E_i$, one has 
	\begin{equation*}
		\capacity(E) \leq \sum_{i=1}^\infty \capacity(E_i)\,. 
	\end{equation*}
		\item \label{incstab}{\em Stability with respect to increasing sequences of sets:} Let $E_i$, $i=1,2,\ldots,$ be arbitrary subsets of $\Omega_\infty$ with the property $E_1 \subset E_2 \subset \cdots.$ and denote $E = \cup_{i=1}^\infty E_i$. Then 
	\begin{equation*}
		\lim_{i \to \infty} \capacity(E_i) = \capacity(E)\,. 
	\end{equation*}
		\item \label{decstab}{\em Stability with respect to decreasing sequences of compact sets:} Let $K_i \subset \Omega_\infty$, $i=1,2,\ldots$, be a decreasing sequence of compact sets and denote $K = \cap_{i=1}^\infty K_i$. Then 
	\begin{equation*}
		\lim_{i \to \infty} \capacity(K_i) = \capacity(K)\,. 
	\end{equation*}
	\item \label{openrep} Let $U \Subset \Omega_\infty$ be an open set. Then 
	\begin{equation*}
		\capacity(U) = \mu_{R_{U}}(\Omega_\infty)\,. 
	\end{equation*}
	\end{enumerate}
\end{theorem}

\begin{proof}
	From the methods developed in \cite{KKKP} we see that \eqref{subadd} and \eqref{incstab} follow from Lemma \ref{ComparisonPrinciple}. 
	Property \eqref{decstab} is a consequence of Theorem \ref{thm:extremal} and Lemma \ref{Stability2}. Property \eqref{openrep} follows from \eqref{incstab}, Theorem \ref{thm:extremal}, and Lemma \ref{Convergence} as in \cite[Lemma 5.9]{KKKP}.
\end{proof}

In conclusion we have established more than enough to say that Borel sets are Choquet capacitable:

\begin{theorem}
	The PME capacity is Choquet capacitable (inner regular). This means that for all Borel sets $E \subset \Omega_\infty$ it holds that 
	\begin{equation*}
		\capacity(E) = \sup \{\capacity(K) : K \subset E, K \text{ compact} \}\,. 
	\end{equation*}
\end{theorem}
\begin{proof}
	Since the capacity is monotone, stable with respect to increasing sequences of sets (Theorem \ref{thm_prop}, \eqref{incstab}) and stable with respect to decreasing sequences of compact sets (Theorem \ref{thm_prop}, \eqref{decstab}), it is a regular capacity and hence the claim follows from Choquet's capacitability theorem \cite[Theorem 9.3, p. 155]{Choq}.
\end{proof}

\end{document}